\documentclass[aps,12pt,final,notitlepage,oneside,onecolumn,nobibnotes,nofootinbib,superscriptaddress,noshowpacs,centertags]{revtex4}

\usepackage{amsthm}
\usepackage{graphicx,amsmath,amssymb}
\usepackage{dcolumn}
\usepackage{bm}
\usepackage[T2A,T1,T5]{fontenc}
\usepackage[utf8]{inputenc}

\newtheorem{theorem}{Theorem}
\newtheorem{corollary}{Corollary}
\begin{document}

\title{Optimization in a non-linear Lanchester-type model\\ involving supply units}
\author{\firstname{H. N.}~\surname{Nguyen}}
\email{nguyenhongnam1977@gmail.com}
\affiliation{%
Department of Mathematics, Faculty of Information Technology, Le Quy Don Technical University
}%

\author{\firstname{M. A.}~\surname{Vu}}
\email{anhmy7284@gmail.com}
\affiliation{%
Department of Mathematics, Faculty of Information Technology, Le Quy Don Technical University
}%

\author{\firstname{D. M.}~\surname{Hy}}
\email{ducmanhktqs@gmail.com}
\affiliation{%
Department of Mathematics, Faculty of Information Technology, Le Quy Don Technical University
}%

\author{\firstname{N. A.}~\surname{Ta}}
\email{tangocanh@gmail.com}
\affiliation{%
Department of Mathematics, Faculty of Information Technology, Le Quy Don Technical University
}%

\author{\firstname{A. T.}~\surname{Do}}
\email{doanhtuanktqs@gmail.com}
\affiliation{%
Department of Mathematics, Faculty of Information Technology, Le Quy Don Technical University
}%

\begin{abstract}
In this paper, a non-linear Lanchester-type model involving supply units is introduced. The model describes a battle where the Blue party consisting of one armed force $B$ is fighting against the Red party. The Red party consists of $n$ armed forces each of which is supplied by a supply unit.
A new variable called "fire allocation" is associated to the Blue force, reflecting its strategy during the battle.
A problem of optimal fire allocation for Blue force is then studied. The optimal fire allocation of the Blue force allows that the number of Blue troops is always at its maximum. It is sought in the form of a piece-wise constant function of time with the help of "threatening rates" computed for each agent of the Red party. Numerical experiments are included to justify the theoretical results.
\end{abstract}
\maketitle
\section{Introduction}
In 1916, Lanchester (\cite{Lan}) introduced a mathematical model for
a battle in the form of a system of differential equations two
unknowns of which  are the number of the two involved parties. In
1962, Deitchman (\cite{Dei}) extended Lanchester's model by
investigating battle between an army and a guerilla force. This
model is called a guerilla warfare model or an asymmetric model. In
this model, the fire of guerilla force is supposed to be aimed while
of the army is unaimed. Later, Schaffer (\cite{Sch}) and  Schreiber
(\cite{Sch1}) generalized Deitchman's model further by taking into
account the intelligence and considered the problem of optimizing
the fire allocation of the army. Recently, Kaplan, Kress and
Szechtman (KKS) (\cite{Kap},\cite{Kre}) also considered Lanchester
model with intelligence in a scenario of counter-terrorism. In this
asymmetric model, intelligence play a decisive role in the outcome
of the combat. In addition, a lot of researchers are interested in
optimization problems involving warfare models. In 1974, Taylor
(\cite{Tay}) studied several problems of optimizing the fire
allocation for some warfare models. Lin and Mackay (\cite{Mac})
extended Taylor's results on optimization of fire allocation for
Lanchester's model of the form one against many. A common interest
of these two studies is optimizing the number of troops. Feichtinger
and his colleagues (\cite{Fei1}) studied an optimization problem for
KKS model with objective function being the cost of the battle,
control variables being intelligence and reinforcement. In 2019, we
investigated a modified asymmetric Lanchester $(n,1)$ model
describing a combat between a group of $n$ counter-terrorism forces
and a single group of terrorists, see (\cite{Manh}). In these works
above, the role of military supply has not been studied thoroughly.
In a combat, the victory of one party is not only decided by the
armed forces but also by the supply units. In 2017, Kim and his
colleagues (\cite{Kim}) considered a Lanchester's model where one of
the party is supported by a supply unit. Besides, there have been
multi-party combats where a Blue force $B$ fights against $n$
independent forces of the Red party: $R^1, R^2,..., R^n$ and each of
these forces is supported by one of $n$ different supply agents
$A^1, A^2,...,A^n$. The supply units do not fights the Blue force
directly. However, their supports for $R^1, R^2,..., R^n$ may affect
both the progress and the outcome of a combat. One such combat is
the civil war in Syria, starting in 2011. In order to investigate
such war, we introduce a novel model of non-linear Lanchester's
type. In Lanchester's model using system of differential equations,
the rate of troops decreasing of a force is computed by attrition
rate of its rival force multiplied by the rival's number of troops.
In our model, attrition rate of $R^i :\,i=1,2,...,n$ is assumed to
be a linear function of the number of its supply unit and this
supply unit can also be attrited by $B$. Let us recall that in
classical non-linear Lanchester's model, the supply units have not
been taken into account but only the armed forces. Besides, in our
model, a fire allocation is added to the Blue force. This
distribution is used to express the strategy of $B$ during the
conflict. The allocation is in the form of $2n$ non-negative number,
whose sum equals to $1$, multiplied by $B$'s attrition rates with
respect to $R^i$ and $A^i:\,i=1,..,n$. This allocation obviously has
an impact on the dynamics of the system of differential equations
and hence the progress and outcome of the combat. For this novel
model, we consider the problem of optimal fire allocation for $B$,
thus, we seek for fire allocation such that the remaining troops of
$B$ at any time is maximum. The optimal fire allocation is derived
using the so-called "threatening rates", which are computed for
$R^i$ and $A^i:\,i=1,..,n$. Numerical experiments
have justified the theoretical findings. \\
The rest of the paper is
organized as follows.  Section 2 is devoted to present model
settings and to investigate the optimization problem for this model.
Numerical experiments are presented in Section 3 to
illustrate the theoretical results. Conclusion and some possible further developments are discussed in the last section.
\section{Main Results}
\subsection{The Model}
Let us consider a combat between a Blue party and a Red party and
assume that each agent of the Red party is supplied by a corresponding supply
unit.
We use the following notations:
\begin{itemize}
\item$B$: Blue force.
\item$R^i(i=1,\ldots ,n)$: $i$-th agent of the Red party.
\item$A^j(j=1,\ldots ,n)$: corresponding supply agents for $R^j$ forces.
\item$r_{R^i}$: an attrition rate of $B$ to $R^i$.
\item $r_{A^j}$: an attrition rate of $B$ to $A^j$.
\item $f_{\alpha^i}$: an attrition function of $A_i$ complementing $R^i$ to $B$.
\item $P = \left( p_
1,\ldots,p_i,\ldots , p_n,p_{n+1},\ldots,p_{n+j},\ldots , p_{2n} \right)$: the fire allocating proportion of $B$ to $R^i$ and $A^j \, (i,j=1,\ldots , n)$ respectively.
\item $\alpha^i_c:$ the fully-connected attrition rate of $R^i$ with $A^i$.
\item $\alpha^i_d:$ the fully-disconnected attrition rate of $R^i$ with $A^i$ $(\alpha^i_d\leq\alpha^i_c)$.
\end{itemize}

We denote this model as $\left( {B \texttt{ vs } ((R^1,A^1),\ldots,(R^n,A^n))} \right).$ The diagram for this model is presented in Figure \ref{hinh1}.
\begin{figure}
\includegraphics[scale=1]{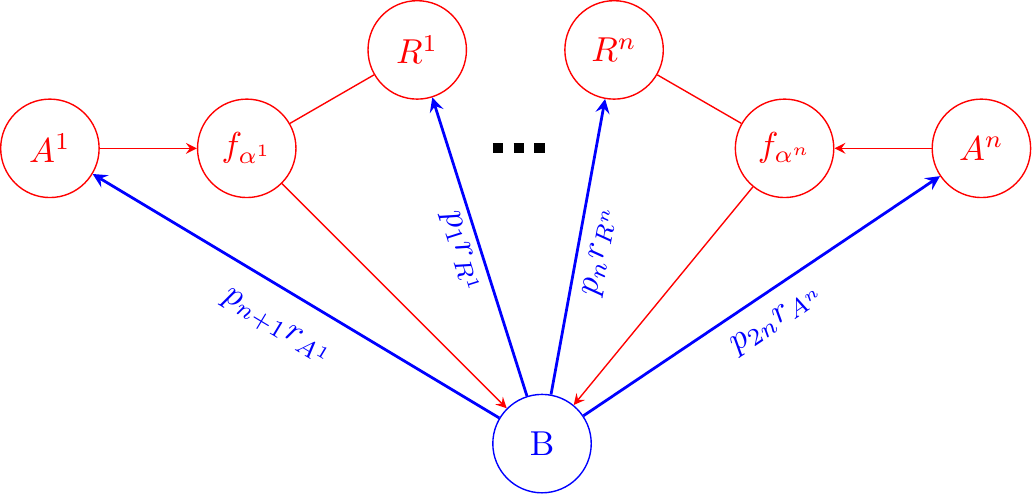}
 \caption{Diagram for the model $\left( {B \texttt{ vs } ((R^1,A^1),\ldots,(R^n,A^n))} \right).$}
 \label{hinh1}
 \end{figure}
Let us consider the problem of finding the optimal fire allocation
of $B$ such that at any time $t$, the remaining troops of $B$ is
maximized. We seek for the optimal fire allocating proportion of $B$
in the form of a piece-wise constant function. This choice is
realistic since it is absurd to alter the fire allocation
constantly, especially during a certain stage of the battle. For
this purpose, we assume that $P = \left(p_ 1,\ldots ,
p_n,p_{n+1},\ldots, p_{2n} \right)$ is a piece-wise constant
proportion where $p_{1},\ldots,p_{2n}\in \left[ {0;1}
\right]:\,\sum\limits_{k=1}^{2n}p_k = 1$  at any time $t$.

The attrition rates of $R^i$ to $B$ is supposed to be entirely
dependent on its corresponding supply unit. Thus the complementing
attrition functions are assumed to be linear ones of the form:
\begin{equation*}
\begin{aligned}
f_{\alpha^i} & = \alpha _d^i + (\alpha _c^i - \alpha
_d^i)\frac{A^i}{A^i_0}: i=1,\ldots,n.
\end{aligned}
\end{equation*}
where  $A^i_0$ number of $A^i$'s troops at the beginning. Let us observe that, at the beginning, when $A^i=A^i_0$, $R^i$ and $A^i$ has a full connection and $f_{\alpha^i}$ attains its maximal value $\alpha _c^i$. When $A^i$ is totally eliminated by $B$, $A^i=0$, the connetion between $R^i$ and $A^i$ is terminated and $f_{\alpha^i}$ is now only $\alpha _d^i$.

 The numbers of troops of all the parties involved in the battle are governed by the following system of differential equations:
\begin{equation}
\begin{cases}
\frac{dB}{dt} =  - \sum \limits_{i = 1}^n \left[ \alpha _d^i + \left( \alpha _c^i - \alpha _d^i \right) \frac{A^i}{A_0^i} \right] R^i,\\
\frac{dR^i}{dt} =  - p_{i} r _{R^i}B: i = 1,\ldots , n,\\
\frac{d A^j}{dt} =  - p _{n+j}r _{A^j}B: j = 1,\ldots , n.
\end{cases}
\label{model}
\end{equation}
 It is apparent that supply agents $A^1,\ldots,{A^n}$ create their impact on the outcome of the battle by influencing the attrition rate of $R^1,\ldots, R^n$ to $B$. When their numbers of troops are reduced to zero, their impacts are stopped accordingly.
 \subsection{Optimal fire allocation of Blue force}
For the model \eqref{model}, we consider the problem of maximizing the Blue force's number of troops at any time $t$.
 Let us compute the following:
 \begin{equation}\label{bi}
\begin{aligned}
b_{i} &= \alpha _c^ir _{R^i} : i=1,\ldots , n, \\
b_{n+j} &= \frac{r _{A^j}\left( \alpha _c^j - \alpha _d^j\right)R^j_0}{A_0^j}:j=1,\ldots , n.
\end{aligned}
\end{equation}
We will refer to these numbers as \textit{"threatening rates"}. These rates represent the "threats" which $R^1,\ldots, R^n$ forces and theirs supply agents expose to the Blue force. The optimal fire allocation of $B$ is pointed out in the following theorem.
\begin{theorem}\label{main}
Suppose that the optimal fire allocation $P^*$ is sought in the set $\mathcal{P} =\{ P=(p_1,\ldots ,p_n,p_{n+1},\ldots ,p_{2n}): p_k \in [0,1] \text{ is constant } \forall k=1,\ldots,2n; \sum \limits_{k=1}^{2n} p_k =1\} $. Then the optimal fire allocation of $B$ is :\\
${P^*} = \left( {0,...,0,\underbrace 1_{k},0,...,0} \right) \,where\,\,k = \mathop {\arg \max }\limits_{l = 1,...,2n} \left\{ b_l \right\}.$
\end{theorem}
\begin{proof}
Let $X(t) = \int _0^t B(s) ds$. It follows that $X' (t) = B(t)$ and
\begin{equation}
\begin{aligned}
X''(t) =B'(t) = - \sum \limits_{i = 1}^n \left[ \alpha _d^i + \left( \alpha _c^i - \alpha _d^i \right) \frac{A^i}{A_0^i} \right] R^i.
\end{aligned}
\label{d2}
\end{equation}
We also have
$$\int _0^t dR^1=-\int_0^tp_{1}r _{R^1}B(s) ds \Rightarrow R^1(t) -R^1(0) = -p_{1}r _{R^1}X(t).$$ This leads to
\begin{equation} \label{R^1t}
R^1(t) = -p_{1}r _{R^1}X(t) + R^1_0.
\end{equation}
By similar arguments, we get
\begin{eqnarray}
R^i(t) = -p_{i}r _{R^i}X(t) + R^i_0: i=1,\ldots,n, \label{R^it}\\
A^j(t) = -p_{n+j}r _{A^j}X(t) + A^j_0:j=1,\ldots,n. \label{A^it}
\end{eqnarray}
Substituting \eqref{R^it}, \eqref{A^it} into \eqref{d2} we obtain
\begin{equation}\label{d2C}
X''(t) = -C_1X^2(t) +C_2X(t) -C_3,
\end{equation}
where
\begin{equation*}
\begin{aligned}
{C_1} &= \sum\limits_{i = 1}^n {\frac{{p_{i}{p_{n+i}}{r _{R^i}}{r _{{A^i}}}\left( {\alpha _c^i - \alpha _d^i} \right)}}{{A_0^i}}}, \\
{C_2} &= \sum\limits_{i = 1}^n {\frac{{{p_{n+i}}{r _{{A^i}}}\left( {\alpha _c^i - \alpha _d^i} \right){R^i_0} + {p_{i}}{r _{R^i}}\alpha _c^iA_0^i}}{{A_0^i}}}, \\
{C_3} &= \sum\limits_{i = 1}^n {\alpha _c^i}{R^i_0}.
\end{aligned}
\end{equation*}
Multiplying both sides of \eqref{d2C} by $dX'(t)$ and integrating, one gets
\begin{equation*}
\begin{aligned}
X'(t) & = B(t) \\
& = \sqrt{-\frac{2}{3}C_1X^3(t) + C_2X^2(t) -2C_3X(t) +C_4},
\end{aligned}
\end{equation*}
where $C_4$ is an integral constant. Since $C_3$ is not changing in time while $C_1,\,C_2\geq 0,$ in order to maximize $B(t)$, we will seek for conditions for which $C_1$ is minimal and $C_2$ is maximal simultaneously. Thus, we consider the multi-objective optimization problem
\begin{equation}\label{moo}
\begin{cases}
\min \limits _{P\in{\mathcal{P}}} C_1,\\
\max\limits _{P\in{\mathcal{P}}} C_2.
\end{cases}
\end{equation}
Let us denote
\begin{equation} \label{ai}
a_i=\frac{r_{R^i}r_{A^i}(\alpha_c^{i}-\alpha_d^{i})}{A_0^i}:i=1,\ldots , n.
\end{equation}
The problem \eqref{moo} now becomes
\begin{equation}\label{moo_xyz}
\begin{array}{l}
\left\{ \begin{array}{l}
\min \sum\limits_{i = 1}^n {a_ix_iy_i}\\
\min -\left( \sum\limits_{i = 1}^n {b_{i}x_i+\sum\limits_{j = 1}^n {b_{n+j}y_j}} \right)
\end{array} \right.{\rm{ }}\\
{\rm{s}}{\rm{.t: }}\left\{ \begin{array}{l}
0 \le {x_i},{y_j} \le 1 :i,\,j=1,\ldots,n,\\
\sum\limits_{i = 1}^n  x_i+\sum\limits_{j = 1}^n  y_j = 1.
\end{array} \right.
\end{array}
\end{equation}
In order to solve the problem \eqref{moo_xyz}, we use the scalarization method. Thus, for each $\gamma \in [0.1]$ we define the function\\
$$\begin{aligned}
F_\gamma\left( {{x_1},\ldots,{x_n},{y_1},\ldots,{y_n}} \right) =  \gamma\sum\limits_{i = 1}^n {a_ix_iy_i} - \left( {1 - \gamma } \right)\left( \sum\limits_{i = 1}^n {b_{i}x_i+\sum\limits_{j = 1}^n {b_{n+j}y_j}} \right),
\end{aligned}$$\\
and consider the following problem
\begin{equation}\label{e12}
\begin{array}{l}
\min {F_\gamma\left( {{x_1},\ldots,{x_n},{y_1},\ldots,{y_n}} \right)}\\
{\rm{s}}{\rm{.t: }}\left\{ \begin{array}{l}
0 \le {x_i},{y_j} \le 1 :i,j=1\ldots,n,\\
\sum\limits_{i = 1}^n  x_i+\sum\limits_{j = 1}^n  y_j = 1.
\end{array} \right.
\end{array}
\end{equation}
By substituting $x_1=1-\sum\limits_{i =2 }^n {x_i}-\sum\limits_{j = 1}^n {y_j}$ into \eqref{e12} we obtain the following problem:
\begin{equation}
\begin{array}{l}
\min {F_\gamma }(1-\sum\limits_{i =2 }^n {x_i}-\sum\limits_{j = 1}^n {y_j},x_2,\ldots,x_n,y_1,\ldots,y_n)\\
{\rm{s}}{\rm{.t: }}\left\{ \begin{array}{l}
0 \le {x_i},{y_j} \le 1 :i=2\ldots,n,\,j=1,\ldots,n,\\
\sum\limits_{i =2 }^n {x_i}+\sum\limits_{j = 1}^n {y_j}\leq 1.
\end{array} \right.
\end{array}
\end{equation}
Without loss of generality, assume that: $b_2\geq \ldots \geq b_n\geq b_{n+1}\geq \ldots \geq b_{2n}.$\\
We consider the following cases:
\begin{itemize}
\item[1.] $b_1=\mathop {\min }\limits_{i = 1,...,2n} \left\{ b_i \right\}$\\
We have:
\begin{equation*}
\begin{aligned}
F_\gamma &=  \gamma\left( a_1y_1\left( 1-\sum\limits_{i =2 }^n {x_i}-\sum\limits_{j = 1}^n {y_j}\right) +\sum\limits_{i = 2}^n {a_ix_iy_i}\right)\\
 &-  \left( {1 - \gamma } \right)\left( b_1+\sum\limits_{i = 2}^n \left( b_i-b_1\right) x_i+\sum\limits_{j = 1}^n \left( {b_{n+j}-b_1}\right) y_j\right).\\
\end{aligned}
\end{equation*}

Since $\gamma\left( a_1y_1\left( 1-\sum\limits_{i =2 }^n {x_i}-\sum\limits_{j = 1}^n {y_j}\right) +\sum\limits_{i = 2}^n {a_ix_iy_i}\right)\geq 0$, one gets
\begin{equation*}
\begin{aligned}
\min F_\gamma \geq & - \left( {1 - \gamma } \right)\left( b_1+\sum\limits_{i = 2}^n \left( b_i-b_1\right) x_i+\sum\limits_{j = 1}^n \left( {b_{n+j}-b_1}\right) y_j\right)\\
\geq & - \left( {1 - \gamma } \right)\left( b_1+\sum\limits_{i = 2}^n \left( {b_2-b_1}\right) x_i+\sum\limits_{j = 1}^n \left( {b_2-b_1}\right) y_j\right)\\
= & - \left( {1 - \gamma } \right)\left( b_1+\left( {b_2-b_1}\right)\left( \sum\limits_{i = 2}^n x_i+\sum\limits_{j = 1}^ny_j\right)\right) \\
\geq & -(1-\gamma)(b_1+b_2-b_1)=-(1-\gamma)b_2 = F_\gamma(0,1,0,\ldots,0).\\
\end{aligned}
\end{equation*}

\item[2.]  $b_1=\mathop {\max }\limits_{i = 1,...,2n} \left\{ b_i \right\}$\\
We have:
\begin{equation*}
\begin{aligned}
F_\gamma &= \gamma\left( a_1y_1\left( 1-\sum\limits_{i =2 }^n {x_i}-\sum\limits_{j = 1}^n {y_j}\right) +\sum\limits_{i = 2}^n {a_ix_iy_i}\right)\\
&- \left( {1 - \gamma } \right)\left( b_1+\sum\limits_{i = 2}^n \left( b_i-b_1\right) x_i+\sum\limits_{j = 1}^n \left( {b_{n+j}-b_1}\right) y_j\right)\\
& =\gamma\left( a_1y_1\left( 1-\sum\limits_{i =2 }^n {x_i}-\sum\limits_{j = 1}^n {y_j}\right) +\sum\limits_{i = 2}^n {a_ix_iy_i}\right)\\
& +(1-\gamma)\left( \sum\limits_{i = 2}^n \left( b_1-b_i\right) x_i+\sum\limits_{j = 1}^n \left( b_1-b_{n+j}\right) y_j\right)-(1-\gamma)b_1. \\
\end{aligned}
\end{equation*}
 Since:
\begin{equation*}
\begin{aligned}
  \gamma\left( a_1y_1\left( 1-\sum\limits_{i =2 }^n {x_i}-\sum\limits_{j = 1}^n {y_j}\right) +\sum\limits_{i = 2}^n {a_ix_iy_i}\right)+(1-\gamma)\left( \sum\limits_{i = 2}^n \left( b_1-b_i\right) x_i+\sum\limits_{j = 1}^n \left( b_1-b_{n+j}\right) y_j\right)\geq 0,\\
\end{aligned}
\end{equation*}
 so:
$\min F_\gamma \geq -(1-\gamma)b_1= F_\gamma(1,0,0,\ldots,0).$
\item[3.]  $b_2\geq \ldots \geq b_k\geq b_1\geq b_{k+1}\geq \ldots \geq b_{2n}.$
\begin{itemize}
\item[3.1.] If $2\leq k<n,$ we have:
\begin{equation*}
\begin{aligned}
F_\gamma &= \gamma\left( a_1y_1\left( 1-\sum\limits_{i =2 }^n {x_i}-\sum\limits_{j = 1}^n {y_j}\right) +\sum\limits_{i = 2}^n {a_ix_iy_i}\right)\\
&+\left( {1 - \gamma } \right)\left( \sum\limits_{i = k+1}^n\left( b_1-b_i\right) x_i+\sum\limits_{j = 1}^n\left( b_1-b_{n+j}\right) y_j\right) -\left( 1-\gamma\right) \left(b_1+\sum\limits_{i =2 }^k\left( b_i-b_1\right) x_i\right).\\
\end{aligned}
\end{equation*}
Since:
\begin{equation*}
\begin{aligned}
 &\gamma\left( a_1y_1\left( 1-\sum\limits_{i =2 }^n {x_i}-\sum\limits_{j = 1}^n {y_j}\right) +\sum\limits_{i = 2}^n {a_ix_iy_i}\right)\\
 &+\left( {1 - \gamma } \right)\left(\sum\limits_{i = k+1}^n\left( b_1-b_i\right) x_i+\sum\limits_{j = 1}^n\left( b_1-b_{n+j}\right) y_j\right)\geq 0,\\
\end{aligned}
\end{equation*}
so:
\begin{equation*}
\begin{aligned}
\min F_\gamma &\geq -\left( 1-\gamma\right) \left(b_1+\sum\limits_{i =2 }^k\left( b_i-b_1\right) x_i\right)\\
&\geq-\left( 1-\gamma\right)\left( b_1+b_2-b_1\right)\\
&= -\left( 1-\gamma\right)b_2=F_\gamma\left( 0,1,0,\ldots,0\right).\\
\end{aligned}
\end{equation*}

\item[3.2.] If $n\leq k \leq 2n,$ we have:
\begin{equation*}
\begin{aligned}
F_\gamma &= \gamma\left( a_1y_1\left( 1-\sum\limits_{i =2 }^n {x_i}-\sum\limits_{j = 1}^n {y_j}\right) +\sum\limits_{i = 2}^n {a_ix_iy_i}\right)+\left( {1 - \gamma } \right)\left(\sum\limits_{j ={k+1} }^n\left( b_1-b_j\right) y_j \right) \\
&-\left( 1-\gamma\right) \left(b_1+\sum\limits_{i = 2}^n\left( b_i-b_1\right) x_i+\sum\limits_{j = 1}^k\left( b_{n+j}-b_1\right) y_j\right).  \\
\end{aligned}
\end{equation*}

Since:\\
$\gamma\left( a_1y_1\left( 1-\sum\limits_{i =2 }^n {x_i}-\sum\limits_{j = 1}^n {y_j}\right) +\sum\limits_{i = 2}^n {a_ix_iy_i}\right)+\left( {1 - \gamma } \right)\left(\sum\limits_{j ={k+1} }^n\left( b_1-b_j\right) y_j \right)\geq 0,$\\
so:
\begin{equation*}
\begin{aligned}
\min F_\gamma &\geq -\left( 1-\gamma\right) \left(b_1+\sum\limits_{i = 2}^n\left( b_i-b_1\right) x_i+\sum\limits_{j = 1}^k\left( b_{n+j}-b_1\right) y_j\right)\\
&\geq-\left( 1-\gamma\right)\left( b_1+b_2-b_1\right)\\
&= -\left( 1-\gamma\right)b_2=F_\gamma\left( 0,1,0,\ldots,0\right).
\end{aligned}
\end{equation*}
\end{itemize}
\end{itemize}
\end{proof}
\begin{corollary} \label{2vs2} For $n=2$ we obtain the model $\left( {B \texttt{ vs } ((R^1,A^1),(R^2,A^2))} \right).$ Then the optimal fire allocation of $B$ is :\\
\begin{equation}
P ^* =
\begin{cases}
(1,0,0,0) \textrm{ if } b_{1}=\mathop {\max }\lbrace b_{1},b_{2},b_{3},b_{4}\rbrace,\\
(0,1,0,0) \textrm{ if } b_{2}=\mathop {\max }\lbrace b_{1},b_{2},b_{3},b_{4}\rbrace,\\
(0,0,1,0) \textrm{ if } b_{3}=\mathop {\max }\lbrace b_{1},b_{2},b_{3},b_{4}\rbrace,\\
(0,0,0,1) \textrm{ if } b_{4}=\mathop {\max }\lbrace b_{1},b_{2},b_{3},b_{4}\rbrace.
\end{cases}
\end{equation}
\end{corollary}
In this case, since the Blue force $B$ actually fights against four
forces $R^1, R^2, A^1, A^2$, the combat theoretically consists of
four stages. These stages are illustrated in Figure \ref{h2}.
\begin{figure}
\centering
\includegraphics[scale=0.7]{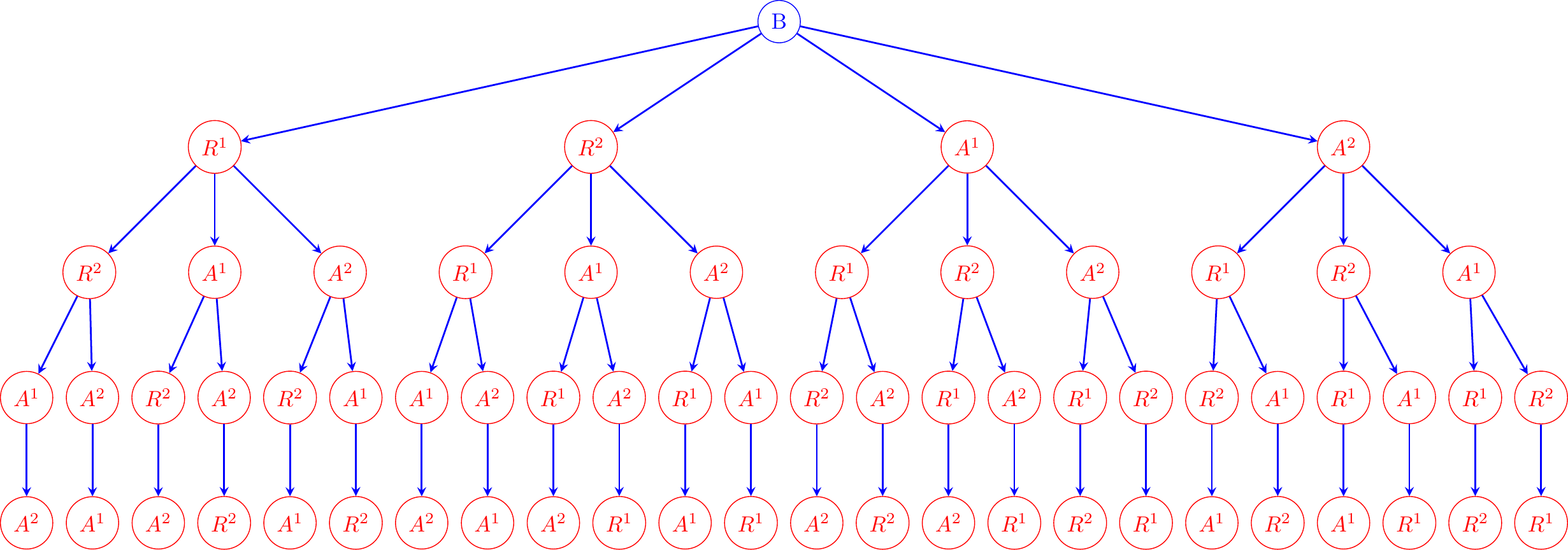}
\caption{Possible progresses of $\left( {B \texttt{ vs } ((R^1,A^1),(R^2,A^2))} \right).$}
\label{h2}
\end{figure}
The Blue force will focus its fire power to whichever possesses the largest threatening rate.
However, let us emphasize that if $R^1$ or $R^2$ is eliminated then
the Blue force doesn't have to fight $A^1$ or $A^2$. Moreover, if
both $R^1$ and $R^2$ are terminated, the combat ends promptly (since
the Blue force is no longer attrited). Therefore, the combat may
actually consists of two or three stages only. Actual stages are
depicted in Figure \ref{h3}.
\begin{figure}[hbtp]
\centering
\includegraphics[scale=0.7]{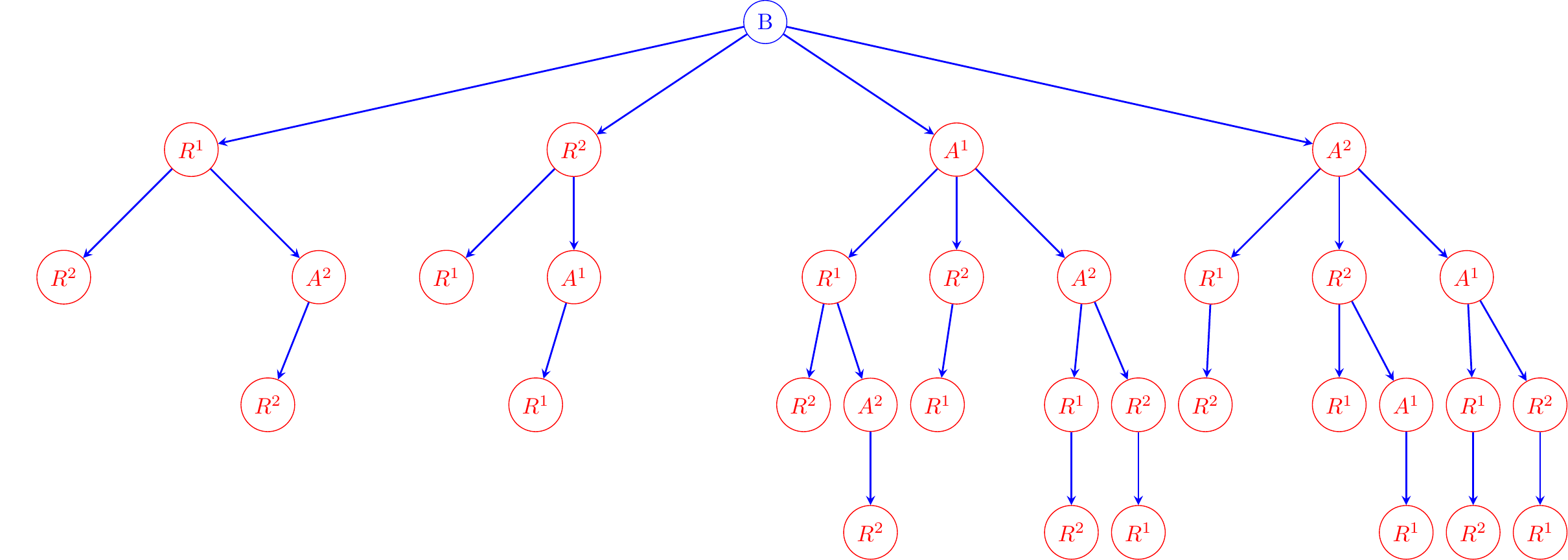}
\caption{Actual stages of the combat $\left( {B \texttt{ vs } ((R^1,A^1),(R^2,A^2))} \right).$}
\label{h3}
\end{figure}

\begin{corollary}\label{1vs1} For $n=1$ our model becomes $\left( {B \texttt{ vs } (R^1,A^1)} \right).$ Then the optimal fire allocation of $B$ is :\\
\begin{equation}
P ^* =
\begin{cases}
(1,0) \textrm{ if } b_{1}\geq b_{2},\\
(0,1) \textrm{ if } b_{1}\leq b_{2}.
\end{cases}
\end{equation}
\end{corollary}
This result has been established in the work of Kim (\cite{Kim}).
\section{ Numerical illustrations}
In this section, we will present numerical experiments for three typical cases of Corollary \ref{2vs2}.
\subsection{Experiment 1: The Armed forces are attacked first}
Let us consider equation \eqref{model} with the following parameters:
\begin{table}[h]
\centering
\begin{tabular}{|c|c|c| }
\hline $(\alpha_d^1, \alpha_c^1)$ &$(\alpha_d^2, \alpha_c^2)$ &
$(r_{R^1}, r_{R^2},r_{A^1},r_{A^2}) $ \\
\hline
$(0.4,0.8)$&$(0.5,0.7)$&$(0.4,0.35,0.3,0.4)$\\
\hline
\end{tabular}
\caption{Parameters for Experiment 1.}
\end{table}\\
together with the following initial conditions: $${{B}_{0}}=350;{R}_{0}^1=120;\,\,R_0^2=100;\,{A}_{0}^1=50;\, A_0^2=60.$$  For these parameters, "threatening rates" are computed as
$$
{b}_1=0.32;\,\,{b}_2=0.245;\,{b}_3=0.288;\,{b}_4=0.133.
$$
Applying Corollary \ref{2vs2}, we conclude that the Blue force will focus its power to attack $R^1$ in the first stage. After the first stage, $R^1$ is eliminated and therefore $A^1$ is not necessarily to be attacked. For the second stage, applying Corollary  \ref{1vs1}, the Blue force will concentrate its fire power to $R^2$. Concluding $R^2$, the battle ends.
To sum up, the optimal fire allocation for $B$ is
$$
P^*=(1,0,0,0) \rightarrow (0,1,0,0).
$$
We compare the optimal fire allocation with the strategy
$${{P }_{1}}=\left(0, 1,0,0 \right)\to \left(1, 0,0,0 \right).$$
The allocation of $P_1$ is similar to $P^*$ but in reverse order. Another allocation used to compare is
$$
P_2=(0,0,1,0)\rightarrow (0,1,0,0) \rightarrow (1,0,0,0).
$$
The progress of the battles  with three fire allocations and number of Blue's troops are reflected in Figure \ref{th1}.
\begin{figure}[hbtp]
\centering
\includegraphics[scale=0.5]{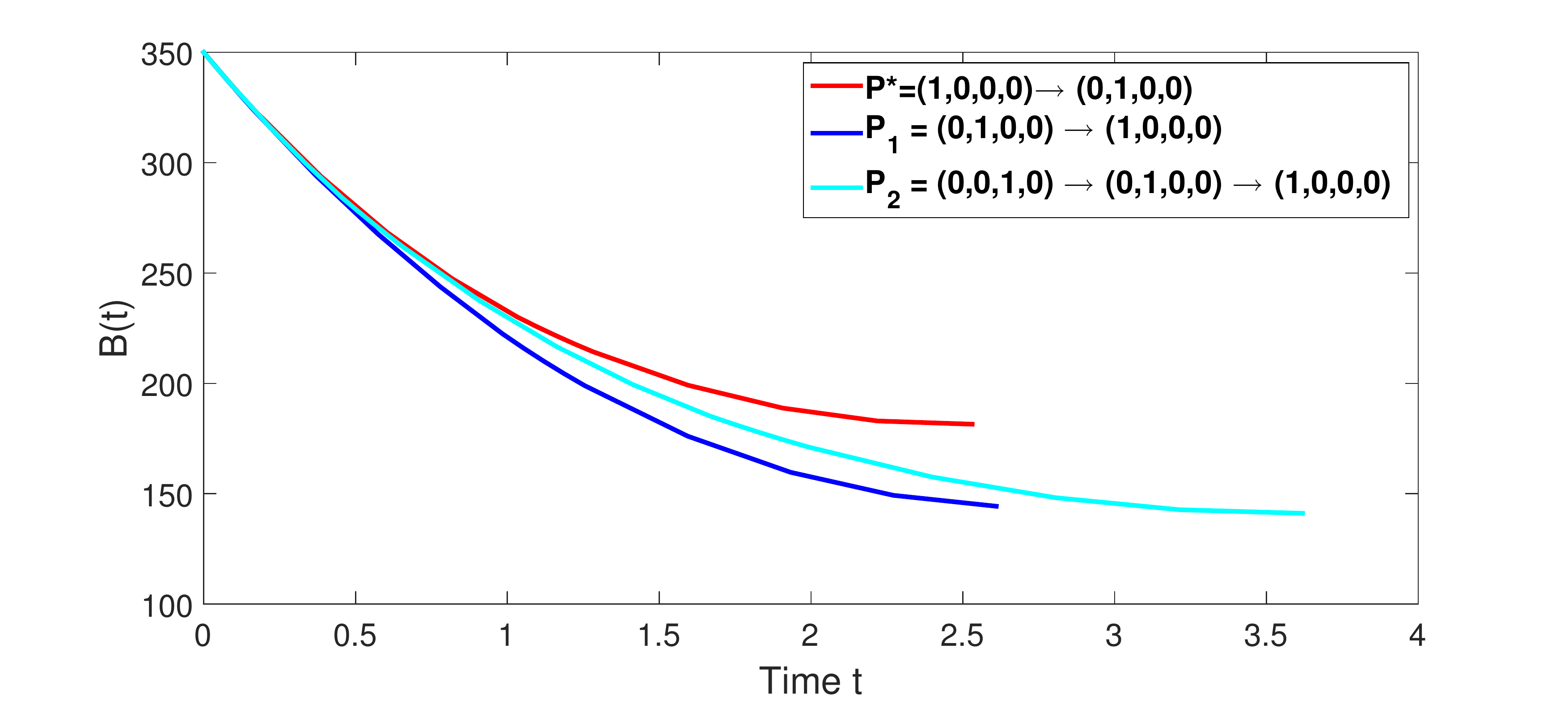}
\caption{Number of Blue's troops vs Time in Experiment 1.}
\label{th1}
\end{figure}
\subsection{Experiment 2: One of the supply agents is attacked first}
Let us consider equation \eqref{model} with the following parameters:
\begin{table}[h]
\centering
\begin{tabular}{|c|c|c| }
\hline $(\alpha_d^1, \alpha_c^1)$ &$(\alpha_d^2, \alpha_c^2)$ &
$(r_{R^1}, r_{R^2},r_{A^1},r_{A^2}) $ \\
\hline
$(0.4,0.8)$&$(0.3,0.7)$&$(0.4,0.5,0.3,0.4)$\\
\hline
\end{tabular}
\caption{Parameters for Experiment 2.}
\end{table}\\
together with the following initial conditions: $${{B}_{0}}=350;{R}_{0}^1=120;\,\,R_0^2=150;\,{A}_{0}^1=50;\, A_0^2=40.$$  For these parameters, "threatening rates" now become
$$
{b}_1=0.32;\,\,{b}_2=0.35;\,{b}_3=0.288;\,{b}_4=0.6.
$$
By Corollary \ref{2vs2}, we derive that the Blue force should concentrate its power to attack $A^2$ in the first stage. After the first stage, $A^2$ is excluded, and the threatening rates must be recalculated as follows:
$$
b_1=0.32; \, b_2=0.15;\,b_3=0.288.
$$
From this observation, the strategy for the second stage is concentrating firepower to $R^1$. When $R^1$ is got rid of, $A^1$ doesn't need to be eliminated and $R^2$ must obviously be targeted in the third stage.
In conclusion, the optimal fire allocation for $B$ is
$$
P^*=(0,0,0,1) \rightarrow (1,0,0,0) \rightarrow (0,1,0,0).
$$
We now measure the differences between $P^*$  and two fire allocations
$${{P }_{1}}=\left(1,0,0,0 \right)\to \left(0,1,0,0 \right)$$
and
$${{P }_{2}}=\left(0,0,1,0 \right)\to \left(1,0,0,0 \right).$$
The allocation of $P_1$ is similar to $P^*$ but it does not include an offence against $A^2$. On the other hand, $P_2$ indicates that Blue force will choose to attack supply agent $A^1$ instead of $A^2$ in the first stage and then focus on $R^1$. The advancement of the battles  depicted in Figure \ref{th2} demonstrate that the Blue force can still win with strategy $P_1$ but it loses with strategy $P_2$.

\begin{figure}[hbtp]
\centering
\includegraphics[scale=0.5]{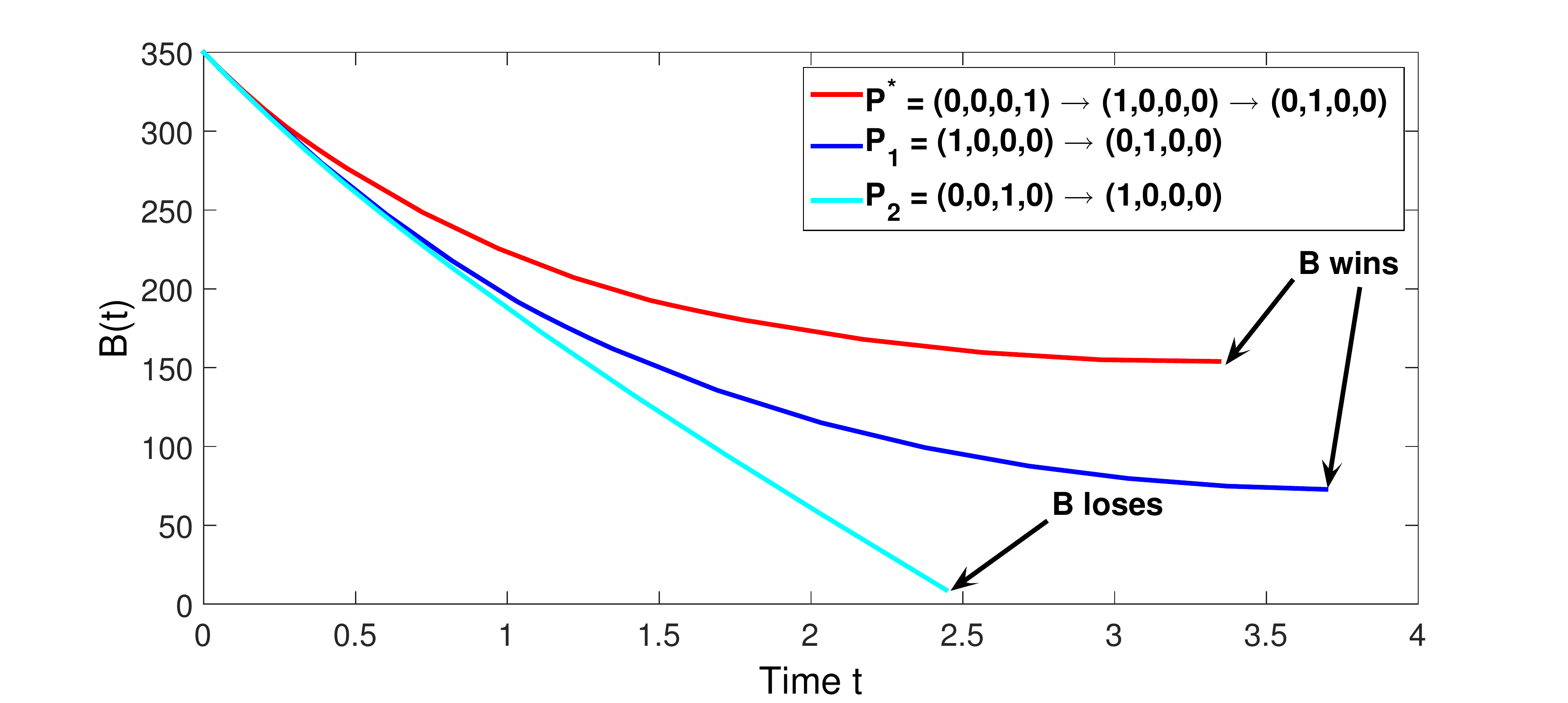}
\caption{Number of Blue's troops vs Time in Experiment 2.}
\label{th2}
\end{figure}

\subsection{Experiment 3: Both of the supply agents are attacked in early stages}
The following parameters are now use with our model:
\begin{table}[h]
\centering
\begin{tabular}{|c|c|c| }
\hline $(\alpha_d^1, \alpha_c^1)$ &$(\alpha_d^2, \alpha_c^2)$ &
$(r_{R^1}, r_{R^2},r_{A^1},r_{A^2}) $ \\
\hline
$(0.4,0.8)$&$(0.3,0.7)$&$(0.2,0.5,0.3,0.4)$\\
\hline
\end{tabular}
\caption{Parameters for Experiment 3.}
\end{table}\\
together with the following initial conditions: $${{B}_{0}}=400;{R}_{0}^1=120;\,\,R_0^2=150;\,{A}_{0}^1=50;\, A_0^2=40.$$
Threatening rates from \eqref{bi} now become
$$
{b}_1=0.16;\,\,{b}_2=0.35;\,{b}_3=0.288;\,{b}_4=0.6.
$$
It follows that the Blue force should concentrate its power to attack $A^2$ in the first stage.
After the first stage, $A^2$ is extinguished, and the recalculated threatening rates are as follows:
$$
b_1=0.16; \, b_2=0.15;\,b_3=0.288.
$$
The supply agent $A^1$ will therefore be targeted in the second stage.
After two stages, supply agents are left out of the picture and the armed force $R^1$ should be targeted in the third stage since it possesses higher threatening rate.
In conclusion, the optimal fire allocation for $B$ is
$$
P^*=(0,0,0,1) \rightarrow (0,0,1,0) \rightarrow (1,0,0,0) \rightarrow (0,1,0,0).
$$
In order to justify the optimality of $P^*$, the following fire allocations will be brought into comparison
$${{P }_{1}}=\left(0,0,1,0 \right)\rightarrow \left(1,0,0,0 \right) $$
and
$${{P }_{2}}=\left(0,0,0,1 \right)\rightarrow (0,1,0,0) \rightarrow (0,0,1,0) \rightarrow \left(1,0,0,0 \right).$$
The allocation of $P_1$ is opposing to $P^*$ as it indicates that the Blue force will target $A^1$ instead of $A^2$ in the first stage and the armed force $R^1$ instead of a supply agent in the second stage. Whereas, $P_2$ is similar to the optimal fire allocation $P^*$. It demonstrates that Blue force will choose to attack supply agent $A^2$  in the first stage, just like $P^*$, and then focus on $R^2,A^1,R^1$ in the later stages, respectively. The development of the battles portrayed in Figure \ref{th3} demonstrate that the Blue force can still win with strategy $P_2$ but it loses with strategy $P_1$.
\begin{figure}[hbtp]
\centering
\includegraphics[scale=0.5]{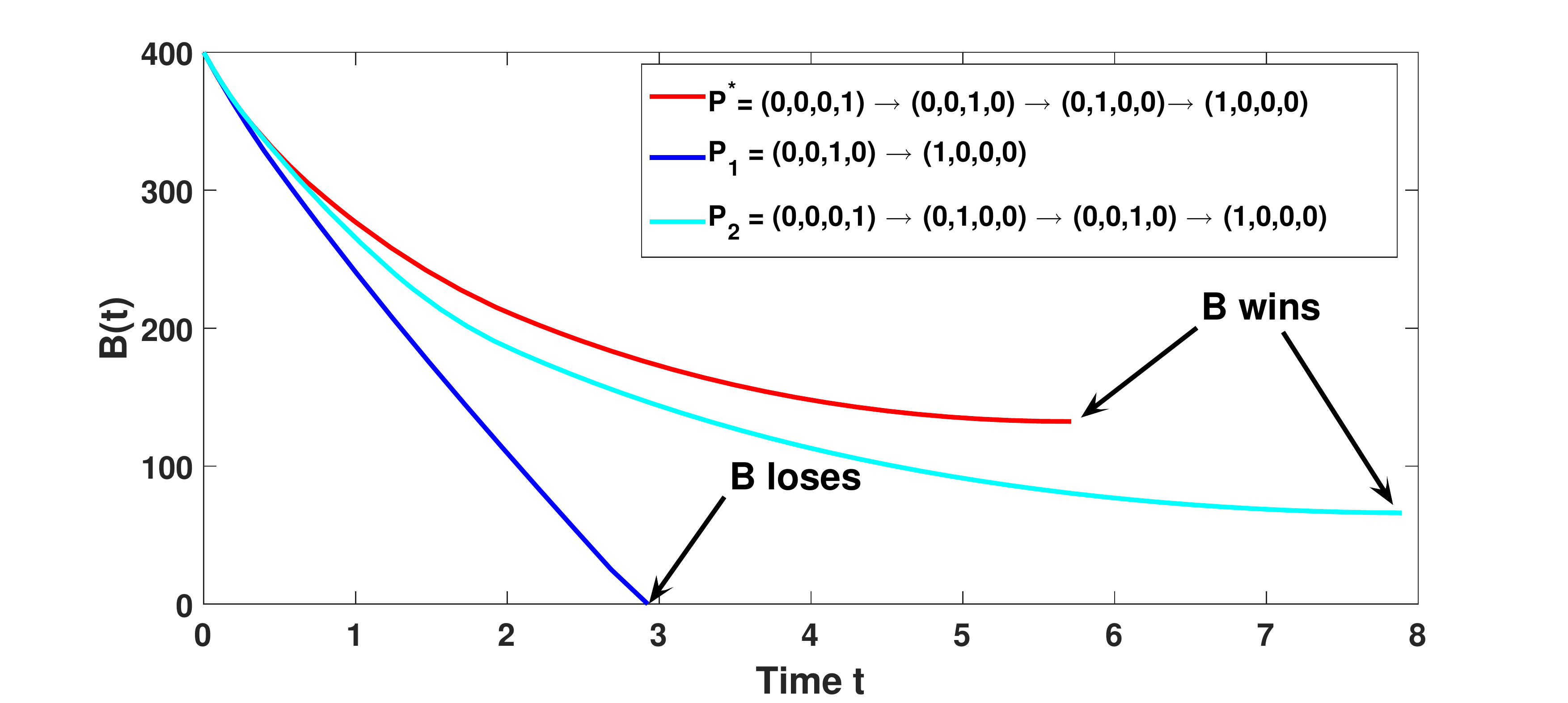}
\caption{Number of Blue's troops vs Time in Experiment 3.}
\label{th3}
\end{figure}
\section{Conclusion}
In this work, we have introduced a novel model for battle with
supplies. Computing the "threatening rates" of the Red party's
entities, we managed to show the optimal fire allocation of the Blue
party. Basically, the Blue force will target entity possessing the largest "threatening rate".
These results have improved and generalized some known
results in this field.
\section*{References}
\nocite{*}

\end{document}